\documentclass[11pt]{article}
\usepackage{amsmath,amssymb,amsthm, latexsym,color,epsfig,enumerate,a4, graphicx}
\usepackage{changepage}
\usepackage{tikz}
\usepackage{xspace}
\usepackage{bbding}
\usepackage[ruled,vlined, linesnumbered]{algorithm2e} 

\theoremstyle{plain}
\newtheorem{theorem}{Theorem}[section]
\newtheorem{theorem*}{Theorem}
\newtheorem{lemma}[theorem]{Lemma}

\newtheorem{corollary}[theorem]{Corollary}

\newtheorem{question}{Question}

\theoremstyle{definition}

\newcommand{\eps}{\ensuremath{\varepsilon}}

\newcommand{\Gnp}{\ensuremath{\mathcal G(n,p)}}

\newcommand{\calG}{\ensuremath{\mathcal G}}

\newcommand{\crchi}{\ensuremath{\chi_{\mathrm{cr}}(H)}}

\date{}
\title{\vspace{-0.7cm} On Koml\'os' tiling theorem in random graphs}
\author{
	Rajko Nenadov 
	\thanks{
		School of Mathematical Sciences, Monash University, Melbourne, Australia. Email: {\tt rajko.nenadov@monash.edu}.
	} 
	\and
	Nemanja \v Skori\'c
	\thanks{
		Institute of Theoretical Computer Science, ETH Zurich, 8092 Zurich, Switzerland. Email: {\tt nskoric@inf.ethz.ch}.
	}
}

\begin{document}
\maketitle

\begin{abstract}
Conlon, Gowers, Samotij, and Schacht showed that for a given graph $H$ and a constant $\gamma > 0$, there exists $C > 0$ such that if $p \ge Cn^{-1/m_2(H)}$ then asymptotically almost surely every spanning subgraph $G$ of the random graph $\Gnp$ with minimum
degree at least 
$$
	\delta(G) \ge \left(1 - \frac{1}{\crchi} + \gamma \right)np
$$
contains an $H$-packing that covers all but at most $\gamma n$ vertices. Here, $\crchi$ denotes the \emph{critical chromatic threshold}, a parameter introduced by Koml\'os. We show that this theorem can be bootstraped to obtain an $H$-packing covering all but at most $\gamma (C/p)^{m_2(H)}$ vertices, which is strictly smaller when $p > C n^{-1/m_2(H)}$. In the case where $H = K_3$ this answers the question of Balogh, Lee, and Samotij.
Furthermore, we give an upper bound on the size of an $H$-packing for certain ranges of $p$.
\end{abstract}

\section{Introduction}

Given graphs $G$ and $H$, a family of vertex-disjoint copies of $H$ in $G$ is called an \emph{$H$-packing}. This generalises the notion of matchings from edges ($H = K_2$) to arbitrary graphs. The study of sufficient degree conditions on $G$ which enforce the existence of a \emph{perfect} $H$-packing (an $H$-packing which covers all vertices of $G$), usually referred to as an \emph{$H$-factor}, dates back to the seminal work of Corr\'adi and Hajnal \cite{hajnal} and Hajnal and Szemer\'edi \cite{hajnal70sz}. In particular, they showed that every graph with $n = \ell k$ vertices and minimum degree at least $(\ell-1)n/\ell$ contains an $K_\ell$-factor. %This was later extended by Hajnal and Szemer\'edi \cite{hajnal70sz} who showed that $\delta(G) \ge (1 - 1/\ell)n$ suffices for the existence of a $K_\ell$-factor (assuming obvious divisibility conditions). 
Such bound on the minimum degree is easily seen to be the best possible. 

Progress towards generalising this result to an arbitrary graph $H$ was made in \cite{alon1992almosth,alon1996h,komlos2001proof}. The approximate result was obtained by Koml\'os \cite{komlos2000tiling}, where he determined the best possible bound on the minimum degree which enforces an $H$-factor covering all but at most $o(n)$ vertices. In particular, he showed that the main parameter which governs the existence of such a packing is the so-called \emph{chromatic threshold} $\crchi$, defined as
$$
	\frac{(\chi(H) - 1)v(H)}{v(H) - \sigma(H)},
$$
where $\sigma(H)$ denotes the minimum size of the smallest colour class in a colouring of $H$ with $\chi(H)$ colours.

\begin{theorem}[Tiling theorem \cite{komlos2000tiling}] \label{thm:tiling_komlos}
	For every graph $H$ and a constant $\gamma > 0$ there exists $n_0 \in \mathbb{N}$ such that if $G$ is a graph with $n \ge n_0$ vertices and 
	$$
		\delta(G) \ge \left(1 - \frac{1}{\crchi} \right) n,
	$$
	then $G$ contains an $H$-packing which covers all but at most $\gamma n$ vertices.
\end{theorem}

% Koml\'os conjectured that the same condition gives an $H$-packing which covers all but at most $C(H)$ vertices, for some constant $C$ depending only on $H$. This was confirmed by Shokoufandeh and Zhao \cite{zhao03koml}. Finally, K\"uhn and Osthus \cite{kuhn2009minimum}  showed that the minimum degree which guarantees a perfect $H$-packing is either determined by $\chi(H)$ or $\crchi$ and gave a complete characterisation for both cases. We refer the reader to \cite{kuhnosthus,kuhn2009minimum} for a detailed survey on the history of the problem and results not mentioned here.

Theorem \ref{thm:tiling_komlos} was further strengthened by Shokoufandeh and Zhao \cite{zhao03koml} and the problem was fully solved only recently by K\"uhn and Osthus \cite{kuhn2009minimum}. We refer the reader to \cite{kuhnosthus,kuhn2009minimum} for a detailed survey on the history of the problem and results not mentioned here.

\subsection{Packing theorems in random graphs}

% Even though the result of K\"uhn and Otshus almost completely settles the problem, it is  interesting to see what happens if we know some structural properties of $G$ rather than just its minimum degree. For example, the conjecture of Fischer \cite{fischer1999variants} states that if $G$ is a balanced $\ell$-partite graph (with $n$ vertices in each class) such that each vertex has at least $(1 - 1/\ell)n$ neighbours in each class other than its own, then $G$ contains a perfect $K_\ell$-packing. Note that it gives significantly weaker bound on the minimum degree compared to the Hajnal-Szemer\'edi theorem, subject to a certain structural property. After a series of papers, Keevash and Mycroft \cite{keevash2015multipartite} showed that the conjecture is true for all such graphs except the ones given by the construction of Catlin \cite{catlin}.

In this paper we are interested in up to which degree the stated theorems hold in random graphs. In particular, we consider the binomial random graph model $\Gnp$. The obvious question is for which $p$ does $\Gnp$ a.a.s \footnote{Asymptotically almost surely, i.e. with probability going to $1$ as $n \rightarrow \infty$.} contain a perfect $H$-packing. The case where $H = K_2$ was already proven by Erd\H{o}s and Renyi \cite{erdos60evo} and the general case was fully resolved by Johansson, Kahn, and Vu \cite{johansson2008factors}. Once this is settled, in the spirit of previously mentioned results it is natural to study whether subgraphs of random graphs with sufficiently large minimum degree contain a perfect (or almost-perfect) $H$-packing.

% It turns out that, once we have the right tools, the bound on the minimum degree analogue to the one in Koml\'os' tiling theorem (Theorem \ref{thm:tiling_komlos}) transfers to random graphs in a `straightforward' way. One of the most important tools in the extremal graph theory is \emph{Szemer\'edi's regularity lemma} \cite{szeRegularity}. A version of the regularity lemma suitable for sparse graphs, observed independently by Kohayakawa \cite{kohayakawa1997szemeredi} and R\"odl (unpublished), has proven to be a useful tool for tackling extremal problems in the random graph setting. However, the lack of an \emph{embedding lemma} which would cover the whole range of $p$ has limited its success until recently. Such a lemma suitable for application in the random graph setting was stated by Kohayakawa, \L uczak, and R\"odl \cite{kohayakawa1997onk} and it took almost 20 years until it was fully proven by Balogh, Morris, and Samotij \cite{balogh2015independent} and, independently, Saxton and Thomason \cite{saxton2015hypergraph}. Somewhat different version was obtained by Conlon, Gowers, Schacht, and Samotij \cite{conlon2014klr} and in the same paper the authors gave the following theorem as an application. They only stated it for $H = K_\ell$ and remarked that the same proof works for any $H$.
It turns out that, once we have the right tools, the bound on the minimum degree analogue to the one in Koml\'os' tiling theorem (Theorem \ref{thm:tiling_komlos}) transfers to random graphs in a `straightforward' way. 
The right tools turn out to be the sparse version of Szemer\'edi's regularity 
lemma observed by Kohayakava \cite{kohayakawa1997szemeredi} and  R\"odl 
(unpublished) together with the KLR Conjecture, first stated in 
\cite{kohayakawa1997onk} and  proven much later by by Balogh, Morris, and 
Samotij \cite{balogh2015independent} and, independently, Saxton and 
Thomason \cite{saxton2015hypergraph}. Somewhat different version was 
obtained by Conlon, Gowers, Schacht, and Samotij \cite{conlon2014klr} and 
in the same paper the authors gave the following theorem as an 
application. They only stated it for $H = K_\ell$ and remarked that the 
same proof works for any $H$.

\begin{theorem} \label{thm:approximate_packing}
	For any graph $H$ which contains a cycle and a constant $\gamma > 0$, there exist constants $b, C > 0$ such that if $p \ge Cn^{-1/m_2(H)}$, where
	$$
		m_2(H) = \max\left\{ \frac{e(H') - 1}{v(H') - 2} \colon H' \subseteq H \; \text{ and } \; v(H') \ge 3 \right\},
	$$
	then with probability at least $1 - e^{- b n^2 p}$ the random graph $\Gamma \sim \Gnp$ has the property that every spanning subgraph $G \subseteq \Gamma$ with minimum degree 
	$$
		\delta(G) \ge \left( 1 - \frac{1}{\crchi} + \gamma \right) np
	$$
	contains an $H$-packing which covers all but at most $\gamma n$ vertices.
\end{theorem}

% % It was remarked in \cite{balogh2012corradi} that Theorem \ref{thm:approximate_packing} can be obtained for $H = K_3$ using standard arguments of the regularity method (which is also the strategy used in \cite{conlon2014klr}), albeit with no proof. At the time, the aforementioned K\L R conjeture was verified only for a handful of cases, including $K_3$. 
% \NS{Ja bih rekao da je tvrdnja za trouglove prilicno ocigledna (kad jednom imas Steger, Gerke teoremu) i da nije vredno spomena ovde.}

It is known that for $p \ll n^{-1/m_2(H)}$ a.a.s there exists a spanning graph $G \subseteq \Gnp$ with $\delta(G) = (1 - o(1))np$ which does not contain a copy of $H$. Therefore, the bound on $p$ in Theorem \ref{thm:approximate_packing} is the best possible even if we only want to cover a linear fraction of all the vertices. Moreover, the constructions which show the optimality of $\delta(G)$ in Theorem \ref{thm:tiling_komlos} also show that by reducing the constant factor in the minimum degree to $1 - 1/\crchi - \varepsilon$ one cannot hope to cover more than $(1 - \varepsilon) n$ vertices and thus the leftover can not be an arbitrarily small linear fraction. %To summarise, there are two places where Theorem \ref{thm:approximate_packing} can be improved: the extra $\gamma np$ term in the bound on the minimum degree and the number of vertices which are not covered. 

Getting rid of the linear leftover in  Theorem \ref{thm:approximate_packing} seems to be a difficult task. Huang, Lee, and Sudakov \cite{huang2012bandwidth} showed that for constant $p$ and minimum degree at least $(1 - \chi(H) + \gamma) np$ one obtains a.a.s. a perfect $H$-packing if $H$ contains a vertex which does not belong to $K_3$, and otherwise there exists an $H$-packing covering all but at most $O(p^{-2})$ vertices. Moreover, they showed that the bound on the number of leftover vertices in the latter case is optimal up to the constant factor. Significantly improving the bound on $p$, Balogh, Lee, and Samotij \cite{balogh2012corradi} showed that for a constant $\gamma >0$ and $p \ge (C \log n / n)^{1/2}$ a.a.s every spanning subgraph $G \subseteq \Gnp$ with minimum degree $\delta(G) \ge (1/2 + \gamma)np$ contains a $K_3$-packing which covers all but at most $O(p^{-2})$ vertices. The authors further suggested that the $\sqrt{\log n}$ factor in the bound on $p$ is not needed, which we confirm in Theorem \ref{thm:main}. Recently, Allen, B{\"o}ttcher, Ehrenm{\"u}ller, and Taraz \cite{allen2015local}, relying on a sparse version of the \emph{blow-up lemma},  announced that the result of Huang et al. holds for $p \ge (C \log n / n)^{1/\Delta}$, where $\Delta$ is the maximum degree of a given graph $H$. 

% The proofs from \cite{allen2015local,balogh2012corradi,huang2012bandwidth} rely on a version of the \emph{blow-up lemma} and are inherently very technical. In fact, the work done in \cite{balogh2012corradi} could be considered as one of the first attempts to provide a blow-up lemma for sparse graphs. It is also worth noting \cite{allen2015local,huang2012bandwidth} provide much more general result of which the $H$-packing is a corollary. \NS{nisam siguran da li je ovaj pasus  potreban?}

\subsection{Our contribution}

We give a short proof of the theorem which replaces $\gamma n$ in Theorem \ref{thm:approximate_packing} by $\gamma (C/p)^{m_2(H)}$, which is clearly better for all $p > Cn^{-1/m_2(H)}$. 
% The proof is based on iterated application of Theorem \ref{thm:approximate_packing}, or rather on an easy corollary of it. 
% The main idea, which might be of independent interest, is to apply it on gradually decreasing subgraphs which pairwise intersect on a negligible fraction of vertices. Such property helps us to maintain the sufficient minimum degree in each of them and greedily choose a packing until we are left with $O(p^{-m_2(H)})$ vertices.
% The proof is based on iterated application of Theorem \ref{thm:approximate_packing}, which might be of independent interest.
% The main idea is that for each subgraph $G_1 \subseteq G$ on which  we apply Theorem \ref{thm:approximate_packing}, we have reserved a new random subgraph $G_2 \subseteq G$, vertex disjoint from $G_1$ and such that $|V(G_1)|$ is a small linear fraction of  $|V(G_2)|$. This enables us to take the uncovered vertices $U$ from $G_1$ and put them into $G_2$ without hurting the minimal degree of $G_2$ as  the size of $U$ is negligible in comparison to $|V(G_2)|$. By iterating this procedure the set of uncovered vertices get gradually smaller until  we are left with $O(p^{-m_2(H)})$ vertices.
The proof is based on simple bootstraping of Theorem \ref{thm:approximate_packing}, which might be of independent interest.

\begin{theorem} \label{thm:main}
	For any graph $H$ which contains a cycle and a constant $\gamma > 0$, there exists $C > 0$ such that if $Cn^{-1/m_2(H)} \le p \le (\log n)^{- 1 /(m_2(H) -1)}$ then $\Gamma \sim \Gnp$ a.a.s has the property that every spanning subgraph $G \subseteq \Gamma$ with 
	$$
		\delta(G) \ge \left(1 - \frac{1}{\crchi} + \gamma \right) np
	$$ 
	contains an $H$-packing which covers all but at most $\gamma (C/p)^{m_2(H)}$ vertices.
\end{theorem}

Let us briefly compare our result to that of Allen et al. \cite{allen2015local}. On the one hand, for large values of $p$ our theorem gives weaker bound on the size of the largest $H$-packing whenever $m_2(H) > 2$ or $H$ contains a vertex which does not belong to $K_3$. This difference is the most drastic in the case where $H$ is a bipartite graph, in which case the result of Allen et al. implies the existence of a perfect $H$-packing. On the other hand, our theorem is stronger in the sense that it applies for the whole range of $p$ for which the problem is sensible. The result of Allen et al. requires $p \ge (\log n / n)^{1/\Delta}$ and it is easy to check that for all connected graphs $H$ which contain a cycle, other than $H = K_3$, we have $m_2(H) < \Delta$. In particular, this leaves a gap in the covered range of $p$ for all such graphs.

As a corollary we answer the question of Balogh et al. in the special case where $H = K_3$. As already mentioned, the following result is optimal with respect to all parameters, except for the technical upper bound on $p$.

\begin{corollary} \label{cor:main}
	Given a constant $\gamma > 0$, there exists $C > 0$ such that if $Cn^{-1/2} \le p \le (\log n)^{-1}$ then $\Gamma \sim \Gnp$ a.a.s has the property that every spanning subgraph $G \subseteq \Gamma$ with $\delta(G) \ge (2/3 + \gamma) np$ contains a $K_3$-packing which covers all but at most $\gamma (C/p)^{2}$ vertices.
\end{corollary}

The paper is organised as follows. In the next section we give the proof of Theorem \ref{thm:main}. In Section \ref{sec:upper_bound_cover} we discuss bounds on the number of leftover vertices in different ranges of $p$. In particular, the obtained bounds suggest that Theorem \ref{thm:main} can be improved in many cases. Finally, some further research directions and open problems are discussed in Section \ref{sec:remarks}.

\medskip
\noindent
\textbf{Notation. } Given a graph $G = (V, E)$, we denote with $v(G)$ and $e(G)$ the size of its vertex and edge set, respectively. For a subset $S \subseteq V$ we use the standard notation $G[S]$ to denote the subgraph of $G$ induced by $S$, i.e. the graph with the vertex set $S$ consisting of the edges of $G$ with both endpoints in $S$. A partition of a set is a family of pairwise \emph{disjoint} subsets which cover the whole set. Whenever the use of floors and ceilings is not crucial it will be omitted.

\section{Proof of Theorem \ref{thm:main}}

The proof of Theorem \ref{thm:main} is based on iterated application of the following corollary of Theorem \ref{thm:approximate_packing}.

\begin{lemma} \label{lemma:almost_spanning}
	For any graph $H$ which contains a cycle and a constant $\gamma > 0$, there exists $C > 0$ such that if $Cn^{-1/m_2(H)} \le p \le (\log n)^{-1/(m_2(H) - 1)}$ then $\Gamma \sim \Gnp$ a.a.s has the property that every subgraph $G \subseteq \Gamma$ with $v(G) \ge (C/p)^{m_2(H)}$ and minimum degree $\delta(G) \ge (1 - \frac{1}{\crchi} + \gamma) v(G)p$ contains an $H$-packing which covers all but at most $\gamma v(G)$ vertices.
\end{lemma}
\begin{proof}
	Let $C$ and $b$ be constants given by Theorem \ref{thm:approximate_packing} applied with $H$ and $\gamma$. We may assume that $C > 2/b$. We show that for every subset $S \subseteq V(\Gamma)$ of size $|S| = s \ge (C/p)^{m_2(H)}$, the induced subgraph $\Gamma[S]$ has the property that every spanning subgraph $G \subseteq \Gamma[S]$ with minimum degree $\delta(G) \ge (1 - 1/\crchi + \gamma)sp$ contains an $H$-packing which covers all but at most $\gamma s$ vertices. This clearly implies the lemma.

	From $s \ge (C/p)^{m_2(H)}$ we have $p \ge C s^{-1/m_2(H)}$, thus by Theorem \ref{thm:approximate_packing} the induced subgraph $\Gamma[S] \sim \calG(s, p)$ has the desired property with probability at least $1 - e^{-b s^2 p}$. %, for some constant $b$ depending only on $H$ and $\gamma$. 
	From the upper bound on $p$ we further get
	$$
		s \ge (C / p)^{m_2(H)} \ge C p^{-1} p^{-m_2(H) + 1} \ge Cp^{-1} \log n.
	$$
	Therefore, $\Gamma[S]$ has the described property with probability at least $1 - n^{- 2s}$, which is good enough to handle a union-bound over all possible sets $S$.
\end{proof}

% Note that the conclusion of Lemma \ref{lemma:almost_spanning} is the same as of Theorem \ref{thm:approximate_packing}, however it applies on all subgraphs with $\Omega(p^{-m_2(H)})$ vertices rather than only on the spanning subgraph. Not surprisingly, the proof is almost exactly the same as the proof of Theorem \ref{thm:approximate_packing} given in \cite{conlon2014klr} and relies on the sparse regularity lemma and recently proven K\L R conjecture. For the sake of completeness, we provide it in the appendix.

Having Lemma \ref{lemma:almost_spanning} at hand we describe our proof strategy. 
% Let $\Gamma \sim \Gnp$ and $G \subseteq \Gamma$ be a graph with the required minimum degree. Applying Lemma \ref{lemma:almost_spanning} on $G$ gives an $H$-packing which covers all but at most $\gamma n$ vertices, which we denote by $U$. If we could apply Lemma \ref{lemma:almost_spanning} again on $G[U]$ we would get an $H$-packing which covers all but at most $\gamma^2 n$ vertices. Applying the lemma further on the subgraph of $G$ induced by those vertices we cover all but $\gamma^3 n$ vertices, and so on until we get a subgraph on which we cannot longer apply Lemma \ref{lemma:almost_spanning}. By repeating this sufficiently many times we end up with a subgraph of size $O(p^{-m_2(H)})$, which proves the theorem. However, even in the second step we potentially get stuck as there is no control over the minimum degree of $G[U]$. For all we know the graph $G[U]$ could be empty.  \NS{Mislim da je sledeci paragraf  sasvim dobro opisuje strategiju.}
% We go around this issue using the following idea. 
First, we partition the vertex set of $G$ into  subsets $V_1 \dot\cup \ldots \dot\cup V_q$ of gradually decreasing size, with $V_q = \Theta(p^{-m_2(H)})$ being large enough to satisfy the requirement of Lemma \ref{lemma:almost_spanning}. By doing this at random we make sure that every vertex has `good' degree into every such subset. Now apply Lemma \ref{lemma:almost_spanning} on the largest subset $V_1$ to cover all but at most $\gamma |V_1|$ vertices, denoted by $U_1$. Even though the subgraph $G[U_1]$ might be empty, we know that every vertex in $U_1 \cup V_2$ has good degree into $V_2$. Crucially, if $U_1$ is much smaller than $V_2$, the second largest subset, then the number of neighbours of each $v \in U_1 \cup V_2$ relative to the size of $U_1 \cup V_2$ is negligibly smaller than relative to the size of $V_2$. Since the latter is sufficiently large, by carefully choosing the constants we obtain the required minimum degree of $G[U_1 \cup V_2]$ in order to apply Lemma \ref{lemma:almost_spanning}. This way we obtain an $H$-packing of $G[U \cup V_2]$ which covers all but at most $\gamma (|U_1| + |V_2|) \le 2 \gamma |V_2|$ vertices, denoted by $U_2$, and recall that all the vertices in $V_1 \setminus U$ are already covered. Now we repeat the same on the subgraph $G[U_2 \cup V_3]$ to obtain an $H$-packing of $G[V_1 \cup V_2 \cup V_3]$ which covers all but at most $2\gamma |V_3|$ vertices, and so on until we cover all but at most $2\gamma |V_q|$ vertices in $G$. We now make this precise.

\begin{proof}[Proof of Theorem \ref{thm:main}]
	Let $C$ be a constant given by Lemma \ref{lemma:almost_spanning} applied with $H$ and $\gamma/20$ (as $\gamma$). We show that if $\Gamma \sim \Gnp$ satisfies the property of Lemma \ref{lemma:almost_spanning} with these parameters then every spanning subgraph $G \subseteq \Gamma$ with the required minimum degree contains an $H$-packing which covers all but at most $\gamma (C/p)^{m_2(H)}$ vertices. Since the previous happens a.a.s for $p$ as stated, this proves the theorem. For the rest of the proof we let $G \subseteq \Gamma$ be an arbitrary spanning subgraph with $\delta(G) \ge (1 - 1/\crchi + \gamma)np$.

	Let $q \in \mathbb{N}$ be the largest integer such that $n/2^{q-1} > \lceil (C/p)^{m_2(H)} \rceil$ and consider a random partition of $V(G)$ into subsets $V_1, \ldots, V_q$ with $|V_i| = \lfloor n/2^i \rfloor$ for $i \in \{1, \ldots, q - 1\}$. Observe that 
	$$
		|V_q| = n - \sum_{i < q} \lfloor n / 2^i \rfloor \ge n - n \sum_{i < q}2^{-i} = n / 2^{q-1},
	$$
	and similarly $|V_q| \le n / 2^{q-1} + q$. Therefore, from $p \le (\log n)^{-1/(m_2(H) - 1)}$ we obtain
	\begin{equation} \label{eq:Vq_size}
		|V_i| \ge n / 2^{q-1} - 1\ge (C/p)^{m_2(H)} \ge C p^{-1} \log n
	\end{equation}
	for every $i \in [q]$. The expected number of neighbours of each vertex $v \in V(G)$ in $V_i$ is at least $(1 - 1/\crchi + \gamma)|V_i|p$, thus it follows from the Chernoff's inequality for hypergeometric distributions that
	$$
		\Pr\left[ \deg_{G}(v, V_i) \le \left(1 - \frac{1}{\crchi} + \gamma/2\right)|V_i|p \right] = e^{-\Omega(|V_i|p)} \stackrel{\eqref{eq:Vq_size}}{<} 1/n^2.
	$$
	(In the last inequality we assumed $C$ is sufficiently large.) A simple application of a union-bound shows that there exists a partition $V(G) = V_1 \cup \ldots \cup V_q$ with sizes as stated above such that for each $v \in V(G)$ and each $V_i$ we have
	\begin{equation} \label{eq:degree}
		\deg_{G}(v, V_i) \ge \left( 1 - \frac{1}{\crchi} + \gamma/2 \right) |V_i|p.
	\end{equation}

	Our plan is to inductively find an $H$-packing of $G[V_1 \cup \ldots \cup V_i]$ for $1 \le i \le q$ which covers all but at most $\gamma |V_i| / 10$ vertices. Calculation similar  to the one in \eqref{eq:Vq_size} shows that
	$$
		|V_q| \le n/2^{q-1} + q \le 2(C/p)^{m_2(H)} + \log n \le 3(C/p)^{m_2(H)},
	$$
	where in the second inequality we used the maximality of $q$ and an implicit assumption that $n$ is sufficiently large. Therefore, such an $H$-packing for $i = q$ covers all but at most $\gamma |V_q| / 10 \le \gamma(C/p)^{m_2(H)}$ vertices of $G$ which proves the theorem.

	For $i = 1$ we get the desired packing by simply applying Lemma \ref{lemma:almost_spanning} on $G[V_1]$. This is indeed possible since $|V_1| \ge (C/p)^{m_2(H)}$ (see \eqref{eq:Vq_size}) and the minimum degree holds by \eqref{eq:degree}. Note that we obtain slightly larger packing than needed (i.e. we cover all but at most $\gamma |V_1| / 20$ vertices).

	Next, let us suppose that there exists such an $H$-packing for some $i < q$ and let $U \subseteq V_1 \cup \ldots \cup V_i$ denote the subset of vertices which are not covered. Then $|U| \le \gamma |V_i| / 10 \le \gamma |V_{i+1}| / 4$, and for every vertex $v \in V(G)$ we have
	\begin{multline*}
		\deg_{G}(v, U \cup V_{i+1}) \ge \deg_{G}(v, V_{i+1}) \stackrel{\eqref{eq:degree}}{\ge} \left( 1 - \frac{1}{\crchi} + \gamma/2 \right) |V_{i+1}|p \\
		\ge \left( 1 - \frac{1}{\crchi} + \gamma/2 \right) \frac{|U| + |V_{i+1}| }{1 + \gamma/4}p \ge \left( 1 - \frac{1}{\crchi} + \gamma/4 \right) |U \cup V_{i+1}|p.
	\end{multline*}
	In particular, this implies 
	$$
		\delta(G[U \cup V_{i+1}]) \ge \left(1 - \frac{1}{\crchi} + \gamma/4\right)|U \cup V_{i+1}|p,
	$$
	and, as $|V_{i+1}| \ge (C/p)^{m_2(H)}$ (see \eqref{eq:Vq_size}), we can apply Lemma \ref{lemma:almost_spanning} to obtain an $H$-packing of $G[U \cup V_{i+1}]$ which covers all but at most $\gamma |U \cup V_{i+1}| / 20 \le \gamma |V_{i+1}| / 10$ vertices. Since all vertices in $\bigcup_{j \le i}V_j \setminus U$ are already covered by the packing obtained for $i$, this gives the desired $H$-packing of $G[V_1 \cup \ldots \cup V_{i+1}]$.
\end{proof}

\section{A lower bound on the number of leftover vertices} \label{sec:upper_bound_cover}

% The bound on the number of the leftover vertices in Theorem \ref{thm:main} is rather natural. It roughly corresponds to the size of a smallest induced subgraph $G' \subseteq V(\Gamma)$ on which we can hope to apply Theorem \ref{thm:approximate_packing} in order to get a large $H$-packing. This allows us to find the desired $H$-packing using a `local' (or greedy) approach without caring how the remaining vertices will be handled. However, for all smaller subgraphs the packing fails for the density reasons -- by removing a small fraction of edges one can destroy all the copies of $H$. Intuitively, this implies that in order to beat the bound obtained in Theorem \ref{thm:main} one needs a `global' strategy.  \NS{ne razumem bas ovaj parafraf}

% That being said, there is no reason to believe that our bound is tight in the regime of $p$ not covered by the result from \cite{allen2015local}. In this section we give a general upper bound on the number of leftover vertices for an arbitrary graph $H$ and derive a more precise bound in the case of complete graphs. \NS{ne bih rekao da je preciznije za klike. Mozda je i jedan i druga tvrdnja najbolja moguca ili jednako losa.}

The bound on the number of  leftover vertices in Theorem \ref{thm:main} asymptotically matches the lower bound obtained by Balogh et al. \cite{huang2012bandwidth} in the case of triangles. As we will see shortly, the situation is quite different for arbitrary graphs. In this section we obtain a general lower bound on the number of leftover vertices for an arbitrary graph $H$ and compare it with the result from Theorem \ref{thm:main}.

% Balogh et al. \cite{balogh2012corradi} showed that if $Cn^{-1/2} \le p \ll 1$ then $\Gnp$ contains a spanning subgraph with large minimum degree such that $\Omega(p^{-2})$ of its vertices are not contained in a triangle. Under slightly stronger upper bound on $p$, the following lemma provides a similar statement for an arbitrary graph $H$.

We make use of a concentration inequality by Kim and Vu \cite{kimvu00}, slightly rephrased for our particular application.
\begin{theorem}[Kim-Vu Polynomial Concentration]
\label{thm:kimvu}
Let $\mathcal H = (V(\mathcal H), E(\mathcal H))$ be a $k$-uniform hypergraph on $n$ vertices and let $\{t_i \colon i \in V(\mathcal H)\}$ be a set of mutually independent Bernoulli random variables with $E[t_i] = p$. For a subset of vertices $S \subseteq V(\mathcal H)$ let us denote 
$$
Y_S = \sum_{\substack{e \in E(\mathcal H)\\ S \subseteq e}} \prod_{i \in e} t_i 
$$
and for $i \in \{0,1 \ldots, k\}$ we set 
$E_i = \max_{S \subseteq V(\mathcal H), |S| = i} E[Y_S].
$
Furthermore, let $E' = \max_{i \ge 1} E_i$ and $E_{\mathcal H} = \max_{i \ge 0} E_i$. Then for any $\lambda > 1$ it holds
$$
\Pr[|Y_{\emptyset} - E_0| > a_k \sqrt{E' E_{\mathcal H}} \lambda^k] < d_k e^{-\lambda} n^{k-1},
$$
where $a_k = 8^k (k!)^{1/2}$ and $d_k = 2e^2$.
\end{theorem}

The following lemma states the main result of the section. The lemma is formulated for a general graph $H$ which contains a cycle and afterwards we give two corollaries for the case of cliques and cycles.
 \begin{lemma} \label{lemma:deleting}
For any graph $H$ which contains a cycle and a constant $\eps > 0$, there exists $c > 0$ such that if
$$
	n^{-1/m_2(H)} \le  p \le  c^2( \log n)^{-\frac{2e(H)}{e(H)-2}} n^{-\frac{v(H) - 3}{e(H) -2}}
$$
then $\Gamma \sim \Gnp$ a.a.s. contains a spanning subgraph $G \subseteq \Gamma$ with $\delta(G) \ge (1 - \eps) np$ such that at least $c / (n^{v(H) -3 } p^{e(H) -1})$ vertices are not contained in a copy of $H$ in $G$.
\end{lemma}
\begin{proof}
Let $c$ be a sufficiently small constant and in particular such that $c \le (2e(H) v(H))^{-2 e(H)}$.
Consider a subset $X \subseteq V(\Gamma)$ of size $c  / (n^{v(H) -3 } p^{e(H) -1})$ and let $w \in V(\Gamma)$ be an arbitrary vertex from $\Gamma$. For a subset of edges $S \subseteq E(\Gamma)$ let $X^w_S$ denote the family of copies of $H$ in $\Gamma$ which contain vertex $w$, intersect the set $X \setminus w$ on at least one vertex and and contain all edges from $S$.  In particular, we denote $X^w := X^w_{S}$, when $S$ is an empty set.

We define an $e(H)$-uniform hypergraph $\mathcal H$ on  the vertex set $E(K_n)$, where a set of vertices in form a hyperedge if the corresponding edges from $\Gamma$ induce a copy from $X^{w}$. Let $Y_S, E_i, E'$ and $E_{\mathcal H}$ be as in Theorem \ref{thm:kimvu} applied to $\mathcal H$ with $e(H)$ (as $k$) and $\binom{n}{2}$ (as $n$). Note that $Y_S$ is nothing more than the number of copies of $H$ in $X^w$ which contain a fixed subset of edges $S \subseteq E(\Gamma)$. In particular, $Y_{\emptyset} = |X^w|$ and  $E_0 = E[|X^w|]$. Simple calculation shows
$$
E_0 = (1 + o(1)) |X^w| \binom{n}{v(H) - 2} p^{e(H)} \ge |X^w|n^{v(H)-2} p^{e(H)} (v(H))^{-v(H)} .
$$
Let us now estimate $E_i$ for $i \ge 1$. Since any subgraph of $H$ with $i$ edges contains at least 
$
\frac{i- 1}{m_2(H)} + 2
$
vertices (follows from the definition of $m_2(H)$) we have 
$
E_i \le n^{v(H) - \frac{i- 1}{m_2(H)} - 2} p^{e(H) - i}$. 
As $p \ge n^{-1/m_2(H)}$ we have that the right hand side of the previous inequality is a non-increasing function in $i$ and thus
$$
E' :=\max_{i \in \{1, \ldots , e(H)\}} E_i \le 
n^{v(H) -  2} p^{e(H) - 1}.
$$
Note that from the upper bound on $p$ we have
$$
|X^w| p  = \frac{c}{n^{v(H)-3} p^{e(H)-2}} \ge \frac{ \log^{2e(H)}n}{c^{2e(H) -3}} \ge c^{-3} \log^{2e(H)}n,
$$
where we used the fact that $H$ contains a cycle and therefore has at least three edges.
From the upper bound on $c$ we further get
\begin{equation}
\label{eq:kimvu}
\frac{E_0}{\sqrt{E_0 E'}} \ge \sqrt{ |X^w| p (v(H))^{-v(H)}} \ge  \sqrt{c^{-3} (\log n)^{2 e(H)}  (v(H))^{- v(H)}}
\ge  (2 e(H) \log n )^{e(H)}.
\end{equation}
From Kim-Vu Polynomial Concentration Theorem, together with the fact $E_0 \ge E_{\mathcal H} \ge E'$ and \eqref{eq:kimvu}, we obtain a constant $C > 0$ such that 
$$
\Pr[  |X^w| \ge C E[|X^w|] \le e^{- 2 \log n}.
$$
Taking a union-bound over all vertices, this implies that a.a.s for every vertex $w \in V(\Gamma)$ the number of copies of $H$ which contain $w$ and intersect $X^w \setminus w$ is at most 
$C |X^w| n^{v(H) - 2}p^{e(H)}$.
Recall that $|X^w| = c/ (n^{v(H) -3}p^{e(H) -1})$, thus by setting $c := \min\{ \eps / C, (2e(H) v(H))^{-2e(H)}\}$ we have that $|X^w|$ is at most $\eps np$ for each vertex $w \in V(\Gamma)$.

Finally, we obtain the graph $G \subseteq \Gamma$ (initially set $G := \Gamma$) by iterating the following procedure until there remains no copy of $H$ in $G$ which intersects $X$: 
Let $H'$ be a copy of $H$ which intersects $X$ and let $\{v,w\} \in E(H')$ be an edge such that 
$H' \in X^v \cap X^w$. Delete edge $\{v,w\}$ from $G$. Such edge must exist for the following reason. If $H'$ intersects $X$ on more than one vertex then it is easy to see that any edge would do. In the case when $H'$ intersects $X$ on exactly one vertex, then since $H$ contains a cycle there must exist an edge which is disjoint from $X$ and this edge satisfies the desired property.
% Let $H'$ be a copy of $H$ which intersects $X$ at a vertex denoted with $w$ and let $v \in V(H') \setminus w$ be an arbitrary vertex from $H'$ different from $w$. 
% If $H'$ intersects $X$ in more than one vertex then delete an arbitrary edge of $H'$ incident to $v$. Otherwise, we delete an edge of $H'$ incident to $v$ which is disjoint from $X$. Such edge must exist as minimal degree of $H'$ is at least two.

The procedure clearly stops and we just need to show that the remaining graph has the desired minimal degree. Let $v$ be a vertex and $e$ an arbitrary deleted edge incident to $v$. Since $e$ was originally contained in an $H$-copy from $X^v$, from  the fact that $|X^v| \le \eps n p$ we conclude that at most $\eps n p$ such edges incident to $v$ are deleted.
This finishes the proof.
\end{proof}

We state some corollaries of Lemma \ref{lemma:deleting}. In the following corollary, which follows directly from Lemma \ref{lemma:deleting}, we use $C_t$ to denote a cycle on $t$ vertices.

\begin{corollary}[Cycles] \label{cor:cycle}
Given an integer $t \ge 3$ and a constant $\eps > 0$, there exists a positive constant $c > 0$
such that if 
$$
	n^{-1/m_2(C_t)} \le p \le c^2 \log^{-2t/(t-2)} n^{-\frac{t - 3}{t-2}}
$$ 
then $\Gnp$ a.a.s. contains a spanning subgraph of minimum degree at least $(1 - \eps)np$ such that at least
$$
	c / (n^{t-3}p^{t-1})
$$
vertices are not contained in $C_t$.
\end{corollary}

When $t = 4$ the upper bound on $p$ in the previous corollary is $n^{-1/2} \log^{-4} n$. On the other hand, Allen et. al \cite{allen2015local} showed that if 
 $p \gg (\log n / n)^{1/2}$ then every subgraph of $\Gnp$ with minimum degree at least $(1/2 + o(1)) np$  contains a $C_4$-packing which covers all the vertices (assuming the divisibility constraint). In the remaining regime of $p$, Theorem \ref{thm:main} gives a $C_4$-packing which covers all but at most $O(p^{-m_2(C_4)}) = O(p^{-3/2})$ vertices, whereas Corollary \ref{cor:cycle} shows that no $C_4$-packing can have a leftover smaller than $\Omega(1 / (np^3))$. In particular, this leaves a gap when 
 $ n^{-2/3} \le p \le n^{-1/2}$. For cycles of bigger length the result from \cite{allen2015local} still applies only for $p \gg (\log n / n)^{-1/2}$ while the bound on $p$ in Corollary \ref{cor:cycle} becomes smaller.

Next claim is corollary of Lemma \ref{lemma:deleting} applied to the case of complete graphs. 
\begin{corollary}[Complete graphs]
Given an integer $t \ge 3$ and a constant $\eps > 0$, there exists $c > 0$
such that if 
$$
	n^{-1/m_2(K_t)} \le p \ll 1
$$
then $\Gnp$ a.a.s. contains a spanning subgraph of minumum degree at least $(1 - \eps)np$ such that at least
$$
	\max_{\ell \in \{3, \ldots, t\}}\{ c / (n^{\ell - 3}p^{\ell (\ell - 1)/ 2 - 1})\}
$$
vertices are not contained in a copy of $K_t$.
\end{corollary}
\begin{proof}
Let us for simplicity denote $m(t) := \max_{\ell \in \{3, \ldots, t\}}\{ 1 / (n^{\ell - 3}p^{\ell (\ell - 1)/ 2 - 1})\}$.
We prove the claim by induction on $t$. For the base of induction, i.e. $t = 3$, the claim follows directly from the result of Balogh et. al. \cite{balogh2012corradi} (c.f. Proposition 4.6)  in the case when $p \ge C n^{-1/m_2(K_3)}$, for some large enough constant $C$, and by Lemma \ref{lemma:deleting} when $n^{-m_2(K_3)} \le p \le  C n^{-1/m_2(K_3)}$. Let us now assume that the corollary holds for all $t' < t$ and $t\ge 4$.
Note that
$$
\frac{1}{n^{t - 3}p^{t (t - 1)/ 2 - 1}} \le \frac{1}{p^2}
$$
when $p^{t(t-1)/2 - 3} n^{t- 3} \ge 1$. However, this conditions is true when $p \ge n^{-1/m_2(K_{t+1})}$ since 
\begin{align*}
\frac{t (t -1) /2  -3 }{t -3 } = \frac{t^2 - t - 6}{2(t- 3)} = 
\frac{(t +2  )(t -3)}{2(t -3)} = \frac{(t+2) (t -1)}{2(t-1)} = \frac{(t+1)t - 2}{2(t-1)}  = m_2(K_{t+1}).
\end{align*}
This implies $m(t-1) = m(t)$, when $p \ge n^{-1/m_2(K_{t+1})}$. Therefore, by induction hypothesis there exist a constant $c > 0$ such that $\Gnp$ a.a.s  contains a spanning subgraph with minimum degree at least $(1-\eps) np$ such that at least $c \cdot  m(t)$ vertices are not contained in a copy of $K_{t-1}$ and, consequently, in a copy of  $K_t$.
On other hand, if $n^{-1/m_2(K_t)} \le p \le n^{-1/m_2(K_{t+1})}$ then we can use Lemma \ref{lemma:deleting} applied to $H = K_t$ to obtain the result, since 
$
\frac{e(K_t) -2}{v(K_t)-3} > m_2(K_{t+1}).
$
\end{proof}

Recall that the bound of Balogh et al. \cite{balogh2012corradi} implies that for any $\eps >0$ and $t \ge 4$ there exists a constant $C$ such that if  $p \ge C n^{-m_2(K_t)}$ then $\Gnp$ a.a.s contains a spanning subgraph with minimum degree $(1 - \eps) np$ such that $\Omega(p^{-2})$ vertices do not belong to a copy of $K_t$. The corollary above guarantees a larger set of `isolated' vertices in the certain range of $p$. For example, in the case of $K_4$ and $p$ in the interval  $n^{-m_2(K_4)} \le p \ll n^{-1/3}$ we obtain a spanning subgraph which contains $c/(np^5) \gg 1/p^2$ vertices that are not contained in a copy of $K_4$. On the other hand, Theorem \ref{thm:main} gives the existence of an $H$-packing which covers all but at most $O(p^{-m_2(K_4)}) = O(p^{-5/2})$ vertices. A simple calculations shows that this leaves the gap in such range of $p$. In the case of larger complete graphs the situation becomes even less clear.

\section{Concluding remarks} \label{sec:remarks}

Using a simple bootstrapping approach, we showed that if $p \ge Cn^{-1/m_2(H)}$ then $\Gnp$ a.a.s has the property that every spanning subgraph with the minimum degree at least $(1 - 1/\crchi + \gamma)np$ contains an $H$-packing which covers all but at most $O(p^{-m_2(H)})$ vertices. As observed in \cite{balogh2012corradi} this is the best one can hope for in the case where $H = K_3$, since $\Gnp$ contains a spanning subgraph with minimum degree $(1 - o(1))np$ and a set of $\Omega(p^{-m_2(K_3)})$ vertices which do not belong to a copy of $K_3$. 
% In fact, they show that for any $\eps > 0$ if $p \gg n^{-1/2 + o(1)}$ then a.a.s every spanning subgraph $G \subseteq \Gnp$ such that
% \begin{itemize}
% 	\item the minimum degree of $G$ is at least $(2/3 + \eps) np$, and
% 	\item each pair of vertices have at least $\eps np^2$ common neighbors
% \end{itemize}
% contains the \emph{square} of a Hamilton cycle, i.e. a graph obtained from a cycle on $n$ vertices by adding additional edge between any two vertices of distance 2. Such a graph clearly contains a $K_3$-factor if $3 \mid n$. 
This leads to the following question. 
\begin{question}
Let $t \ge 3$ be an integer, and suppose $p \gg n^{-1/m_2(K_t)}$ and $t \mid n$. Is it true that $\Gamma \sim \Gnp$ a.a.s has the property that every spanning subgraph $G \subseteq \Gamma$ such that
\begin{itemize}
	\item the minimum degree of $G$ is at least $(1 - 1/t + o(1))np$, and
	\item every vertex is contained in $\eps n^{t - 1} p^{\binom{t}{2}}$ copies of $K_t$
\end{itemize}
contains a $K_t$-factor?
\end{question} 

Our result from Corollary \ref{cor:cycle} shows that for the case of $C_4$
some leftover is unavoidable when $n^{-1/2 - o(1)}$, while 
 the result of Allen et al. shows that the minimum degree is indeed sufficient for the existence of a $C_4$-factor when $p \gg n^{-1/2}$. However, for $t \ge 5$ their result has the same lower bound on $p$, while Corollary \ref{cor:cycle} shows that the leftover is unavoidable only when $p \le n^{-(t-3)/(t-2) - o(1)}$. 
 It is therefore tempting to conjecture that already $p \ge n^{-(t-3)/(t-2)}$ is enough for the existence of a $C_5$-factor. The proof in \cite{allen2015local} relies on a general blow-up lemma \cite{allen2014blow} in which the bound on $p$ heavily depends on the maximum degree of the graph we wish to embed. Proving a better bound on $p$ for such a lemma seems difficult, however, it is plausible that a version tailored for packings of small  cycles might be easier to obtain. Having said this we ask the following question.
 \begin{question}
Let $t \ge 5$ be an integer, and suppose $p \gg n^{-(t-3)/(t-2)}$ and $t \mid n$. Is it true that $\Gamma \sim \Gnp$ a.a.s has the property that every spanning subgraph $G \subseteq \Gamma$, such that
 the minimum degree of $G$ is at least $(1 - 1/\chi(C_t) + o(1))np$,
contains a $C_t$-factor?
\end{question}

\bibliographystyle{abbrv}
\bibliography{refs}

\end{document}